\documentclass{amsart}

\usepackage{todonotes}
\usepackage{url}
\usepackage{amsfonts}
\usepackage{amssymb}
\usepackage{euscript}
\usepackage{amsmath}
\usepackage{amsthm}
\usepackage[matrix, arrow, curve]{xy}
\usepackage{mathrsfs}
\usepackage{tikz-cd}
\usepackage{dsfont}

\newcommand{\Li}{L_\infty}
\newcommand{\LL}{\mathop{L}}

\newcommand{\T}{T}

\newcommand{\HH}[2]{\mathop{CH^\bullet_{e_{#1}}} (#2, #2)}
\newcommand{\RHom}[3]{\mathop{RHom^\bullet_{#1}}(#2, #3)}
\newcommand{\F}{F}
\newcommand{\R}[2]{{\mathbb R}^{#1}_{#2}}
\newcommand{\Hp}{{{\mathbb R}^2_{>0}}}

\newcommand{\C}{\mathscr{C}}
\newcommand{\SC}{\mathscr{SC}}
\newcommand{\Sh}[1]{\mathcal{F}_{#1}}

\newcommand{\set}[1]{\underline{\mathbf{#1}}}
\newcommand{\We}[2]{\mathop{\mathcal{W}^{#1}}(#2)}

\newcommand{\FM}{\mathbf{FM}_{n}}
\newcommand{\SCO}{\mathbf{SC}_{n}}

\newcommand{\Dil}[1]{\mathrm{Dil}(#1)}
\newcommand{\fm}[1]{\mathfrak{fm}_{#1}}

\newcommand{\CS}[2]{\C(#1)(#2)}

\newcommand{\ScS}[2]{\SC(#1)(#2)}

\newcommand{\Cs}[2]{\C^{0}(#1)(#2)}
\newcommand{\Scs}[2]{\C^{0}(#1)(#2)}

\newcommand{\gl}{SO(n)}

\newcommand{\Disk}{\mathbb{D}}
\newcommand{\Diskc}{\overline{\mathbb{D}}}

\newcommand{\Hom}{\mathop{\mathrm{Hom}}}

\newcommand{\tens}[1]{\mathbin{\mathop{\otimes}\limits_{#1}}}

\renewcommand{\Im}{\mathop{\mathrm{Im}}}
\newcommand{\Pro}{\varPhi}
\newcommand{\obl}[2]{\mathop\mathrm{obl}\nolimits_{#1}^{#2}}

\theoremstyle{plain}
\newtheorem{prop}{Proposition}
\newtheorem{conj}{Conjecture}
\newtheorem{theorem}{Theorem}

\newtheorem{cor}{Corollary}
\theoremstyle{definition}
\newtheorem{definition}{Definition}
\newtheorem{example}{Example}
\newtheorem*{example*}{Example}

\theoremstyle{remark}

\begin{document}

\title{Weyl $n$-algebras and the Swiss cheese operad}
\author{Nikita Markarian}

\date{}

\address{National Research University Higher School of Economics, Russian Federation,
Department of Mathematics, 20 Myasnitskaya str., 101000, Moscow,
Russia}

\email{nikita.markarian@gmail.com}

\thanks{The study has been funded within the framework of the HSE University Basic Research Program and the Russian Academic Excellence Project '5-100'.
}

\begin{abstract}
We apply Weyl $n$-algebras to 
prove   formality theorems for higher Hochschild cohomology.
We present two approaches: via propagators and via the factorization complex.
It is shown that the second  approach is
equivalent to the first one taken with a new family of propagators we introduce.
\end{abstract}

\maketitle

\section*{Introduction}

The present paper continues studies of Weyl $n$-algebras began in \cite{preW, W, WK}.
We describe how these ideas can be applied to prove formality theorems,
which are isomorphisms between higher Hochschild cohomology
of polynomial algebras and Weyl $n$-algebras.
The substantial part of this paper is  rephrasing and generalization
of the pioneer  paper \cite{K}, 
where the formality for usual Hochschild cohomology was firstly proved, 
in terms of Weyl $n$-algebras.

The construction from \cite{K} depends on  choice of a propagator.
There is another approach to  formality via
the factorization homology of Weyl $n$-algebras, 
which was implicitly stated and used in \cite{WK}. 
We show, that for the usual Hochschild cohomology
this formality is equivalent to the one introduced in \cite{K}
but with a different propagator.
Due to the geometric nature of this approach,
 all coefficients of this morphism are rational.
It leads us to a surprising conjecture
that a family of propagators we define
gives formalities with rational coefficients.

Two approaches to the formality described in the present paper
resemble two approaches to the Kontsevich integral of a knot.
The first one using iterated integrals  (see e.~g. \cite[Part 3]{CD})
is similar to the approach via propagator.
The second partly conjectural approach (see \cite{WK}
and references therein) corresponds to the one via the factorization complex.

First three section of the paper  does not contain
any new material.
In the first and second sections, we recall basic definitions
for the present series of papers
of the Fulton--MacPherson operad, Weyl $n$-algebras
and the factorization complex.

The third section
is devoted to the notion
of the Swiss Cheese operad, which was  introduced in
\cite{V,KO}, and the higher Hochschild cohomological complex.
A module over the Swiss Cheese operad
is a triple of an $e_n$-algebra, an $e_{n-1}$-algebra
and some additional data, which is referred to as an action of 
the $e_n$-algebra on the $e_{n-1}$-algebra. 
The main result of \cite{T}
states that given an action of an $e_n$-algebra
on an $e_{n-1}$-algebra, there is a morphism from this $e_n$-algebra
to the higher Hochschild complex of $e_{n-1}$-algebra.
We demonstrate, that for $n=2$
if one takes the  usual Hochschild cohomological complex
as a model for higher Hochschild complex,
the corresponding morphism 
of $L_\infty$-algebras is the one appeared 
in the proof of the formality theorem in \cite{K}.
The Hochschild cohomological complex introduced in 
in \cite{Ger}  is a dg-Lie (not $L_\infty$!) algebra. 
It seems to be an important feature, that 
it is equipped with a pre-Lie algebra
structure, which is not compatible with the differential.
It would be highly interesting to find some explanation 
 of the existence of such a model and 
discover some higher-dimensional generalization of it.
This generalization must be a dg-Lie algebra model of the $L_\infty$-algebra of the higher Hochschild complex.

In the fourth section
we construct  a quasi-isomorphism between the
Weyl $n$-algebra and the higher Hochschild cohomological complex
of the polynomial algebra.
This is a higher-dimensional generalization of the main construction of 
\cite{K} given in terms of Weyl $n$-algebras.
As was mentioned in \cite{K} for $n=2$, this construction works with any propagator. But \cite{K}  and later papers explore merely the same propagator
and its slight variations like in \cite{RW}.
The Conjecture \ref{conj} we formulate  implies
that there are other interesting propagators to work with.

In the fifth section following \cite{WK}
we use the factorization complex to build formality morphisms.
These formalities turn out to be equal to the ones
from the previous section for some particular propagators. 
The terms of these formality morphisms are given by integrals similar to the ones from \cite{AS}.

The key point in the factorization complex approach for an $e_3$-algebra
and 1-sphere (see also \cite{WK})
 is an isomorphism between the Hochschild cohomological
complex of the polynomial algebra and the Hochschild homological
complex of the $e_3$-algebra.
The latter is equipped with the cyclic structure.
But the quasi-isomorphism given by Proposition \ref{sphere} 
does not respect it. Consequently, there is
some interesting 
interaction between this structure and
 the 
formality given by the factorization complex approach.
The paper \cite{WK} may be considered as the first step in studying this interaction.   
Besides as was mentioned in \cite{K} and developed in
later papers (see \cite{RW} and references therein)
the set of formalities for the usual Hochschild cohomological
complex of a polynomial algebra is equipped
with a rich additional structure such as
the Grothendieck--Teichm\"uller
Lie algebra action.  
The interaction of this structure and the structure mentioned
above is a subject for future research.

{\bf Acknowledgments.} I am grateful to  D.~Calaque, V.~Dotsenko,
B.~Feigin,  A.~Khoroshkin, S.~Mer\-ku\-lov,
B.~Shoikhet, D.~Tamarkin and A.~Voronov for fruitful discussions. 

\section{Weyl $n$-algebras}

\subsection{ Fulton--MacPherson operad}

\label{FMC}

Let $\R{n}{}$ be an affine space. For a finite set $S$  denote by $(\R{n}{})^S$ the set of ordered $S$-tuples in $\R{n}{}$.
Let $\Cs{\R{n}{}}{S}\subset (\R{n}{})^S$ 
be the configuration space of distinct ordered points in $\R{n}{}$ labeled by $S$.
In \cite{GJ, Ma} (see also \cite{PS} and \cite{AS}) the Fulton--MacPherson 
compactification 
$\CS{\R{n}{}}{S}$ of $\Cs{\R{n}{}}{S}$ is introduced.
This is a manifold with corners and a boundary with interior
$\imath\colon \Cs{\R{n}{}}{S}\hookrightarrow \CS{\R{n}{}}{S} $.
There is a projection $\pi\colon \CS{\R{n}{}}{S}\to (\R{n}{})^S$
such that $\pi\circ\imath\colon\Cs{\R{n}{}}{S} \to (\R{n}{})^S $ is the natural embedding.
For any $S'\subset S$ there is the projection map
\begin{equation*}
 \CS{\R{n}{}}{S} \to \CS{\R{n}{}}{S'},
\end{equation*}
which forgets points.

The natural action of the group of affine transformations on $\Cs{\R{n}{}}{S}$ is lifted to
$\CS{\R{n}{}}{S}$. Denote by $\Dil{n}$ its subgroup consisting of dilatations  with positive coefficients and shifts. 
Group $\Dil{n}$ acts freely on $\CS{\R{n}{}}{S}$ and the quotient is isomorphic to
the fiber $\pi^{-1}(\vec{0})$, where $\vec{0}\in(\R{n}{})^S$ is the $S$-tuple
sitting at the origin (see e.~g. \cite[2.2]{W}).
Denote any of these isomorphic manifolds by $\FM^{S}$.
The sequence $\FM^{S}$ may be equipped with a structure
of an unital operad in the category of topological spaces,
for details see  \cite[2.2]{W} and other references above. 

\begin{definition}
The sequence of topological spaces $\FM^{S}$ with the unital operad structure as above
is called the Fulton--MacPherson operad.
\end{definition}

Given a topological operad, one may produce a dg-operad by
taking complexes of chains of its components.

\begin{definition}
Denote by $\fm{n}$ the operad of $\mathbb R$-chains of $\FM$. 
\label{chains}
\end{definition}

Real numbers appear here are to simplify things, all object
and morphism we shall use may be defined over rationals.
By chains we mean the complex of de Rham currents, that is why we need real chains.
Mostly below we will consider the cooperad of de Rham cochains of $\FM$.
 
\begin{prop}
Operad $\fm{n}$ is weakly homotopy equivalent to $e_n$,
the operad of chains of the little discs operad.
\label{eq}
\end{prop}

\begin{proof}
See \cite[Proposition 3.9]{PS} and   \cite[3.3]{W}.
\end{proof}

Spaces $\FM^{S}$ are acted on by the general linear group, and, in particular,
by its maximal compact subgroup $\gl$, we suppose that a scalar
product on the space is chosen. One may consider the semi-direct
product of $\gl$ and the Fulton--MacPherson operad
and algebras over it.
But we will need only the following special case 
of such algebras.

\begin{definition}[\cite{W}, Definition 3]
We say that a dg-algebra $A$ over $\fm{n}$ is  invariant,
if all structure maps of complexes
$$
\fm{n} \otimes A\otimes\cdots\otimes A\to A
$$
are invariant under the action of group $\gl$ on complexes of operations
of $\fm{n}$.
\label{inv}
\end{definition}

Note that we mean invariance on the level of complexes, not up to homotopy.

\subsection{Weyl $n$-algebras}
\label{weyl}

The algebras over operad $\fm{n}$ we need below are  Weyl $n$-algebras.
Recall its definition, which  slightly differs from the one given in \cite{W}.
The difference is in the quantization parameter $h$: 
in the mentioned paper we considered algebras over formal series of $h$
since below we suppose that $h=1$. 

Let $n>1$ be a natural number.
Let $V$ be a $\mathbb{Z}$-graded  finite-dimensional vector space  over the base field $\Bbbk$
of characteristic zero containing $\mathbb R$ equipped with a non-degenerate  skew-symmetric pairing 
$\omega\colon V\otimes V\to \Bbbk$
of degree $1-n$. Let $\Bbbk[V]$ be the polynomial algebra generated by $V$.
Denote by 
\begin{equation}
\partial_\omega \colon \Bbbk[V]\otimes \Bbbk[V] \to \Bbbk[V]\otimes \Bbbk[V]
\label{omega}
\end{equation}
the differential operator that is a derivation in each factor
and acts on generators as $\omega$. 

Consider $\FM(\set{2})$,  the space of 2-ary operations 
of the Fulton--MacPherson operad. This is  homeomorphic to the $(n-1)$-dimensional
sphere. Denote by $\mathfrak{v}$ the standard $\gl$-invariant 
$(n-1)$-differential form on it.  For any two-element subset $\{i,j\}\subset S$ 
denote by $p_{ij}\colon \FM(S)\to \FM(\set{2})$ the map that forgets
all points except ones marked by $i$ and by $j$. 
Denote by $\mathfrak{v}_{ij}$ the pullback of $\mathfrak{v}$ under projection $p_{ij}$.
Let $\alpha$ be an element of endomorphisms of
$\Bbbk[V]^{\otimes S} \tens{Aut (S)}C^*(\FM(S))$ (where $C^*(-)$ is the de Rham complex)
given by 
$$
\alpha= \sum_{i,j\in S}\partial_\omega^{ij}\wedge \mathfrak{v}_{ij} ,
$$
where $\partial_\omega^{ij}$ is the operator $\partial_\omega$
applied to the $i$-th and $j$-th factors.
 
\begin{prop}
The composition
$$
\Bbbk[V]^{\otimes S}\stackrel{\exp(\alpha)}{\longrightarrow}\Bbbk[V]^{\otimes S}\otimes C^*(\FM(S))\stackrel{\mu}{\to}\Bbbk[V]\otimes C^*(\FM(S)),
$$
where $\mu$ is the product in the polynomial algebra,
defines an algebra over the operad $\fm{n}$ with the underlying space
$\Bbbk[V]$.
\label{prod}
\end{prop}

\begin{proof}
This is a simple check. 
\end{proof}

The algebra defined in this way is obviously invariant under the action of $\gl$,
thus it is invariant (see Definition \ref{inv}).

\begin{definition}[{\cite{W}}]
For a pair $(V, \omega)$ and $n>1$ as above the invariant $\fm{n}$-algebra given
by Proposition \ref{prod} is called the  Weyl $\fm{n}$-algebra
or the Weyl $n$-algebra. Denote it by $\We{n}{V}$.
\label{wna}
\end{definition}

The  Weyl $1$-algebra is the usual Weyl algebra generated
by a a $\mathbb{Z}$-graded  finite-dimensional vector space $V$ with relations $[x,y]=(x,y)$,
where $x, y\in V$ and
$(\cdot,\cdot)$ is a perfect pairing of degree $0$ on $V$. 
Below we will use this definition only in Proposition \ref{sphere}. 

Note that Proposition \ref{eq} provides us with a notion of  Weyl $e_n$-algebras.

The natural map of of operads $\fm{m}\to \fm{n}$ for $m<n$
induces the functor from $\fm{n}$-algebras to $\fm{m}$-algebras.
As in \cite{WK} denote it by $\obl{n}{m}$.
\begin{prop}
For $m<n$ the $\fm{m}$-algebra $\obl{n}{m} \We{n}{V}$ is isomorphic to the commutative 
polynomial algbebra $\Bbbk[V]$.
\label{obl}
\end{prop}
\begin{proof}
It follows from the very definition of the Weyl $\fm{n}$-algebra.
\end{proof}

\subsection{Lie algebra}

Recall the construction of a morphism from the shifted $L_\infty$ operad to $\fm{n}$,
see e.~g. \cite[2.3]{W}.

Spaces of operations of the Fulton--MacPherson operad are equipped with a stratification
labeled by trees as follows.
As an operad of sets $\FM$ is freely generated  by $\Cs{\R{n}{}}{S}/\Dil{n}$.
Denote by $\mu$ the map from this free operad to the free
operad with one generator in each arity, which sends  generators to generators.
Elements of the latter operad are enumerated by rooted trees. The map above sends
$\C_{\set{k}}^0(\R{n}{})/\Dil{n}$ to the star tree with $k$ leaves.
For a tree $t\in\T(S)$ denote by 
$[\mu^{-1} (t)]\in C_*(\F_n(S))$ the chain presented by its preimage under $\mu$.
The  operad $L_\infty$ is a semi-free operad with generators labeled by trees,
see e.~g. \cite{GK}. One may see that $[\mu^{-1} (\cdot)]$
commutes with differentials. It gives us the following statement.

\begin{prop}
Map $[\mu^{-1} (\cdot)]$ as above gives a morphism 
\begin{equation}
\Li[1-n]\to \fm{n}
\label{morphism}
\end{equation}
from shifted $L_\infty$ operad   
to the $dg$-operad $\fm{n}$ of chains 
of the Fulton--MacPherson operad.
\label{lie}
\end{prop}

\begin{proof}
See e.~g. \cite[Proposition 2]{W}.
\end{proof}

\begin{definition}
For a $\fm{n}$-algebra $A$ call its pull-back under (\ref{morphism})
the  associated $\Li$-algebra and denote it by $\LL(A)$.
\label{la}
\end{definition}

Consider the $\Li$-algebra  $\LL(\We{n}{V})$ associated with
the Weyl $n$-algebra. By the very definition,
all operations on it are given by integration of closed forms by chains of the Fulton--MacPherson
operad. But one may see, that chains representing higher operations
(that is operations, which are not compositions of Lie brackets) in $\Li$
are all homologous to zero, because $\Li$ is a resolution of the Lie operad. 
Thus $\LL(\We{n}{V})$ is a $\mathbb{Z}$-graded Lie algebra, that is all higher operations vanish.
This Lie algebra $\LL(\We{n}{V})$ is a deformation of the Abelian one.
The first order deformation gives the  Poisson Lie algebra: 
the underlying space is the $\mathbb{Z}$-graded commutative algebra
$\Bbbk[V]$, the bracket is defined by $\omega\colon V\otimes V\to k$ 
on generators and
satisfies the Leibniz rule.

\begin{prop}
For $n>1$
Lie algebra $\LL(\We{n}{V})$ is isomorphic to the Poisson Lie algebra
of $(\Bbbk[V], \omega)$.
\label{poiss}
\end{prop}

\begin{proof}
Clear, because for $n>1$ the square of the de Rham cochain $v$ is zero.
\end{proof}

\section{Factorization complex}

\subsection{Factorization complex}
\label{FH}

The factorization complex of an algebra
over the framed $n$-disks operad on a manifold
is the tensor product of the right module
over the framed $n$-disks operad corresponding to the manifold
and the left one defined by the algebra, see e.~g. \cite{G}.
For an invariant $\fm{n}$-algebra
we will use a simplified version of this definition following  \cite{W}.

Let  $M$ be a $n$-dimensional oriented  manifold.
In the same way, as for $\R{n}{}$, there is the
Fulton--MacPherson compactification
$\CS{M}{S}$ of the space $\Cs{M}{S}$ of ordered pairwise distinct points in $M$
labeled by $S$. Locally it is the same thing.
Inclusion $\Cs{M}{S}\hookrightarrow \CS{M}{S}$ is a homotopy equivalence,
there is  a projection $\CS{M}{S} \stackrel{\pi}{\to} M^S$.

Recall that a point in the Fulton--MacPherson compactification $\CS{\R{n}{}}{S}$ of 
the configuration space of $\R{n}{}$
looks like a configuration from the configuration space $\Cs{\R{n}{}}{S'}$
with elements of $\FM$ sitting at each point of the configuration.
It follows that spaces $\CS{\R{n}{}}{\bullet}$ form a right $\FM$-module, 
which  is freely generated by $\Cs{\R{n}{}}{\bullet}$ as a set.
The same is nearly true for the Fulton--MacPherson compactification
of any oriented manifold $M$. But to define such an action one needs to choose
coordinates at the tangent space of any point of a configuration of $\CS{M}{S}$.
To fix it one has to consider either only framed
manifolds or introduce framed configuration space. 
For invariant algebras these problems vanish.

\begin{definition}[{\cite[Proposition 3]{W}}]
For an invariant unital $\fm{n}$-algebra $A$ and an oriented manifold $M$ the
factorization complex $\int_M A$ is
the complex given by the colimit of the diagram
\begin{equation}
\begin{tikzcd}
\bigoplus\limits_{S'} C_*(\CS{M}{S'})\tens{Aut(S')} A^{\otimes S'}\\
\bigoplus\limits_{i\colon S'\to S}C_*(\Cs{M}{S})\tens{Aut(S)}\bigotimes\limits_{s\in 
S} 
(\fm{n}(i^{-1}s)\tens{Aut(i^{-1}s)} A^{\otimes 
{(i^{-1}s)}}) \arrow[u]\arrow[d]\\
\bigoplus\limits_{S}C_*(\Cs{M}{S})\tens{Aut(S)} A^{\otimes S} 
\end{tikzcd}
\label{coend}
\end{equation} 
where the summation in the middle runs over maps between finite sets,
the downwards arrow is given by the left action of $\fm{n}$ on $A$ for
$\Im i$ and the unit for $S\setminus \Im i$
and the upwards arrow is given by the right action of $\fm{n}$ on 
$C_*(\CS{M}{\bullet})$.
\label{colimit}
\end{definition}

Note that relations (\ref{coend})
include in particular colimits
\begin{equation}
\begin{tikzcd}
\bigoplus\limits_{S'} C_*(\CS{M}{S'})\tens{Aut(S')} A^{\otimes S'}\\
\bigoplus\limits_{i\colon S'\hookrightarrow S}C_*(\CS{M}{S})\tens{Aut(S')} A^{\otimes S'} 
\arrow[u]\arrow[d, "{\otimes 1^{(S\setminus S')}}"]\\
\bigoplus\limits_{S}C_*(\CS{M}{S})\tens{Aut(S)} A^{\otimes S} 
\end{tikzcd}
\label{reduced}
\end{equation}
where the upward arrow is the projection, which forgets points labeled by $S\setminus S'$.

\begin{prop}
For a  $\mathbb{Z}$-graded  vector space $V$
the cohomology of the factorization complex $\int_M \Bbbk[V]$ 
of the polynomial algebra $\Bbbk[V]$ is isomorphic to  $\Bbbk[V\otimes H_*(M)]$.
\label{commutative}
\end{prop}

\begin{proof}
The factorization complex of a polynomial algebra $\Bbbk[V]$ on a manifold $M$ is
isomorphic to $\bigoplus_i C_*(M^{\times \set{i}})\otimes_{\Sigma_i} V^{\otimes \set{i}}$
by the very definition.
It follows the statement.
\end{proof}

\begin{prop}
\begin{enumerate}
\item The factorization complex $\int_{M^k} A$
of an invariant $\fm{n}$-algebra 
on a closed compact oriented $k$-manifold $M^k$ is 
naturally equipped with a structure of  $\fm{n-k}$-algebra.
\item For a fiber bundle $E^n\stackrel{F^k}{\to} B^{n-k}$ with 
closed compact oriented base and fiber and an invariant  $\fm{n}$-algebra $A$
$$
\int_{B^{n-k}}(\int_{F^k} A)=\int_{E^n} A,
$$
where $\int_{F^k} A$ is a $\fm{n-k}$-algebra by the previous item.
\label{fubini}
\end{enumerate}
\end{prop}

\begin{proof}
See \cite[Section 5]{GTZ2} and references therein.
\end{proof}

This theorem may be formulated  for maps more
general than projections of fiber bundles.
To define push-forward in a more general situation
one needs to introduce factorization sheaves, see \cite{ATF, G} for details.
The construction from Subsection \ref{action-sec} below is an example
of such a push-forward.

\subsection{Factorization complex of a disk}

The factorization complex is homotopy invariant. In particular, it means,
that the factorization complex of a disk is trivial. It is stated
in two subsequent propositions.

Denote by $\Disk^n$ the open disk  $\{x\in\R{n}{}|\lvert x\rvert<1\}$
and by $\Diskc^n$ the closed disk  $\{x\in\R{n}{}|\lvert x\rvert\le 1\}$. 

\begin{prop}
For a $\fm{n}$-algebra $A$ the factorization  complex $\int_{\Disk^n} A$ is homotopy equivalent to 
$A$ and embedding of any point  $  p \to \Disk^n$  
induces a quasi-isomorphism $A=\int_p A\stackrel{\sim}{\to} \int_{\Disk^n} A$.
\label{odisk}
\end{prop}
\begin{proof}
See e.~g. \cite[Prop. 5]{W}.
\end{proof}

\begin{prop}
For a invariant $\fm{n}$-algebra $A$ the factorization  complex $\int_{\Diskc^n} A$ is homotopy equivalent to 
$A$ and the embedding  $ \Disk^n \to \Diskc^n$  
induces a quasi-isomorphism $\int_{\Disk^n} A\stackrel{\sim}{\to} \int_{\Diskc^n} A$.
\label{module}
\end{prop}

\begin{proof}
The following proof is taken from \cite[5.2]{GTZ}.
Consider the projection $p\colon \Diskc^n\to [0,1]$, which sends point $x$
to $|x|$. The fiber over a non-zero point is the sphere $S^{n-1}$.
The factorization complex of $\int_{S^{n-1}} A$ is a $e_1$-algebra
and $A$ as a complex is a module  over it 
(see \cite{L,F} and also \cite[Proposition 5.8]{GTZ}).
As it follows from gluing property of the factorization complex,  
the factorization complex  $\int_{\Diskc^n} A$
is quasi-isomorphic to 
$$
A\mathop{\otimes}\limits_{\int_{S^{n-1}} A}^L \int_{S^{n-1}} A,
$$
which follows the statement of the proposition.
\end{proof}

From the proof of this proposition it follows that
the factorization complex $\int_{\Diskc^n} A$
is equipped with a structure of ($\int_{S^{n-1}} A$)-module.
As this complex is quasi-isomorphic to $A$
it follows that the underlying complex of $A$ itself is a $\int_{S^{n-1}} A$-module
(\cite[Lemma 5.12]{GTZ}).

\section{Swiss Cheese operad}

\subsection{Swiss Cheese operad}
\label{SC}

Let $\R{n}{}$ be an affine space. 
Denote by $\R{n}{\ge 0}$ and $\R{n}{>0}$
subsets $\{\overrightarrow{x}\in \R{n}{}| x_0\ge 0\}$
and  $\{\overrightarrow{x}\in \R{n}{}| x_0> 0\}$ correspondingly,
where  $x_0$ is the coordinate function.
Denote by $\Scs{\R{n}{\ge 0}}{S}$
the configuration space 
 of distinct ordered points in $\R{n}{\ge 0}$ labeled by $S$. 
Points inside $\R{n}{>0}\subset \R{n}{\ge 0}$ are called closed and points on the boundary
 $\R{n-1}{=0}\subset \R{n}{\ge 0}$
are called open. 
Denote by $\ScS{\R{n}{\ge 0}}{S}$  the closure of $\Scs{\R{n}{\ge 0}}{S}$ in   $\CS{\R{n}{}}{S}$.
This is a manifold with corners and a boundary.
There is a projection $\pi\colon \ScS{\R{n}{\ge 0}}{S}\to (\R{n}{\ge 0})^S$,
which restricts on  $\Scs{\R{n}{\ge 0}}{S}$ to  the natural embedding.

Let us define a stratification of $\Scs{\R{n}{\ge 0}}{S}$
that is a continuous map to a poset.
The poset is ${\{O<C\}}^S$,
where  $\{O<C\}$ is the poset consisting of two elements.
For $s\in S$ the $s$-component of this map is $C$,
if the point of the configuration labeled by $s$ is closed
and is $O$ if it is open.
One may see that taking closures of strata in
$\ScS{\R{n}{\ge 0}}{S}$ defines a Whitney stratification of the latter
space with the same indexing poset.
Denote the indexing map by
\begin{equation}
\varpi\colon \ScS{\R{n}{\ge 0}}{S}\to {\{O<C\}}^S.
\label{strat}
\end{equation}

Denote by $\Dil{n-1}$ the subgroup of affine transformations
of $\R{n}{}$ consisting of dilatations  with positive coefficients and shifts
along the hyperplane $\{x_0=0\}$. 
Group $\Dil{n-1}$ acts freely on $\ScS{\R{n}{\ge 0}}{S}$.
The quotient is isomorphic to
the fiber $\pi^{-1}(\vec{0})$, where $\vec{0}\in(\R{n}{\ge 0})^S$ is the $S$-tuple
sitting at the origin.
It follows that $\pi^{-1}(\vec{0})$ is a retract of $\ScS{\R{n}{\ge 0}}{S}$.
Denote the quotient by $\SCO^{S}$.
Note that  $\pi^{-1}(\vec{x})$ for any $S$-tuple $\vec{x}\in(\R{n}{> 0})^S$ 
in the interior
is isomorphic to $\FM^S$.
%

The sequence of manifolds with corners $\SCO^{S}$ form a colored operad called
the Swiss Cheese operad introduced in \cite{V,KO}.
Describe it as an operad of sets.
This colored operad
has two colors: points may be open and closed.
Note that the set of colors is a poset, that is a category, rather than a set, and there are only
operations compatible with this structure.
This operad of sets is free
and is generated by 
the following operations.
The set of $S$-ary generating operations from $S$ closed points to a close point
equals to the quotient of $\Cs{\R{n}{}}{S}\hookrightarrow\CS{\R{n}{}}{S}$
by $\Dil{n}$, which is embedded in $\FM^S$.
The set of  operations from $C$ closed and $O$ open points to an open point
equals to the quotient of
the configuration space of $C$ distinct points in $\R{n}{>0}$
and $O$ distinct points in $\R{n-1}{=0}$
factored out by the $\Dil{n-1}$ group action.
The action of the symmetric group is straightforward
and the composition is analogous
to the one of the Fulton--MacPherson operad.  

Below we do not need exactly the notion of this colored operad, 
but the action of a $\fm{n}$-algebra on a $\fm{m}$-algebra
 we define below is essentially the
action of this operad.

\subsection{Action}

For a space $X$  with a stratification given by $\varpi \colon X\to P$, where $P$ is
a posetal category, we say that a constructible sheaf with values in a category $\mathrm{C}$ is $\varpi$-combinatorial 
if its restriction to each stratum is constant.
A combinatorial sheaf is defined by a functor $P\to \mathrm{C}$, see e.~g. \cite[1.5]{GK}.

Consider a triple $(A, M, \varepsilon)$ consisting of a unital $\fm{n}$-algebra $A$,
a  $\fm{n-1}$-algebra $M$ and a map of unital $\fm{n-1}$-algebras 
$\varepsilon\colon \obl{n}{n-1} A\to M$.
Denote by $A^{\vee}$ and $M^{\vee}$ the linear dual complexes.
The triple defines a functor  from category $\{O<C\}$ 
to complexes, which sends $C$ to $A^{\vee}$, $O$ to $M^{\vee}$ and $\varepsilon^{\vee}$ to 
the only non-trivial morphism of this category.
The tensor power of this functor gives a functor from
 ${\{O<C\}}^S$  to complexes.
Denote by $\Sh{\varepsilon^{\vee}}$ the combinatorial
sheaf of complexes over $\SC^\bullet$ associated with this functor.

\begin{definition}[{\cite{V, KO, T}}]
\label{sc}
For a triple $(A, M, \varepsilon)$ as above, an action of $A$ on $M$
is defined by 
maps of complexes 
$$
M^{\vee}\to C^*(\SCO^S, \,\Sh{\varepsilon^{\vee}})
$$
such that
\begin{enumerate}
\item(compatibility) their restriction on $\R{n-1}{=0}\subset \R{n}{\ge 0}$
are given by the $\fm{n-1}$-algebra structure on $M$;
\item(factorization) they  factor through the limit of the diagram
\begin{equation*}
\begin{tikzcd}
\bigoplus\limits_{S'} C^*(\CS{\R{n}{\ge 0}}{S'}, \,\Sh{\varepsilon^{\vee}})  \arrow[d]\\
\bigoplus\limits_{i\colon S'\to (C\cup O)}C^*(\Cs{\R{n}{\ge 0}}{C\cup O})
\!\!\!\tens{ Aut(C)\times Aut(O)}\!\!\!
{{\bigotimes\limits_{s\in C} (\fm{n}(i^{-1}s)\tens{Aut(i^{-1}s)} A^{\otimes {(i^{-1}s)}})^{\vee}}
\atop{\bigotimes\limits_{s\in O}{C^*(\SCO^{{i^{-1}} s},\, \Sh{\varepsilon^\vee} )}}} \\
\bigoplus\limits_{C\cup O}C^*(\Cs{\R{n}{\ge 0}}{C\cup O})\tens{Aut(C)\times Aut(O)} ({A}^{\otimes C} \otimes {M}^{\otimes O} )^\vee\arrow[u]
\end{tikzcd}
\end{equation*}
where 
$\Cs{\R{n}{\ge 0}}{C\cup O}$ means the configuration space of
$C$ closed and $O$ open distinct points,
summation in the middle runs over maps between finite sets,
which are surjective on $O$,
the  upwards arrow is given by the left coaction of $\fm{n}$ on $A$ for
$\Im i\cap C$,
action of $A$ on $M$ for $\Im i\cap O$
 and the unit for $S\setminus \Im i$,
and the  downwards arrow is given by the right coaction of $\fm{n}$ on 
$C^*(\CS{\R{n}{\ge 0}}{\bullet})$.
\end{enumerate}
\label{def-action}
\end{definition}

This definition resembles Definition \ref{colimit} of the  factorization complex.
It is not a coincidence, the action may be defined as a factorization
sheaf of a special form, see for details \cite{ATF}, \cite{G}.

\subsection{Higher Hochschild cohomology}
\label{HHC}

Let $M$ be an invariant $\fm{n-1}$-algebra. 
The factorization complex $\int_{S^{n-2}} M$
is an $e_1$-algebra and
the underlying complex
of $M$ is a module over it, see e.~g. \cite{GTZ}
and the remark after Proposition \ref{module}.

\begin{definition}[\cite{F}, \cite{GTZ}]
Define the higher Hochschild cohomological complex of an invariant $\fm{n-1}$-algebra $M$
by 
\begin{equation}
\HH{n-1}{M}=\RHom{\int_{S^{n-2}}M}{M}{M}.
\label{HH}
\end{equation}
\end{definition}

The higher Hochschild complex of an $e_{n-1}$-algebra
is equipped with an $e_{n-1}$-algebra structure.
It may be defined rather explicitly (see \cite{GTZ}):
it is given by the composition of the target of $RHom$.
By \cite{L} and \cite{GTZ} the higher Hochschild cohomology is the
derived centralizer of the identity map 
of a $e_{n-1}$-algebra to itself. By \cite{L} it is equipped
with a canonical  $e_{n}$-algebra structure.
It is shown there that the mentioned $e_{n-1}$ structure on $\HH{n-1}{M}$
may be lifted to a $e_n$-algebra structure.

Action in the sense of Definition \ref{def-action} of an $e_{n}$-algebra  $A$ on 
a $e_{n-1}$-algebra $M$ induces  a morphism of $e_{n}$-algebras 
$$
A\to \HH{n}{A},
$$
see \cite{T}. In terms of the triple $(A, M, \varepsilon)$ it may be defined as follows.
Given a chain in  the complex $ \int_{\Diskc^{n-1}} M$
consider the following chain in the complex dual to $C^*(\SCO^S, \Sh{\varepsilon^{\vee}})$:
its open points are given by this chain,  where $ \Diskc^{n-1}$ is the unit disc
in $\R{n-1}{=0}$ and the only closed point is  $(t, 0, \dots, 0)$
labeled by an element of $A$, where $t\in \R{}{>0}$. Consider the limit of this
configuration as $t$ approaches $0$.
Convolution with the action gives a map
$$
A \to \Hom( \int_{\Diskc^{n-1}} M, M)\backsimeq \Hom\nolimits^\bullet(M,M).
$$ 
One may see, that the resulting element of $ \Hom^\bullet(M,M)$
is a homomorphism of $(\int_{S^{n-2}} M)$-modules.
It gives us  a map
\begin{equation}
A \to \HH{n-1}{M}.
\label{action} 
\end{equation}
  
\begin{prop}
The map (\ref{action}) is a morphism of $e_n$-algebras.
The $e_n$-algebra $\HH{n-1}{M}$ is the final object
in the category of $e_n$-algebras acting on $M$.
\label{thomas}
\end{prop}
\begin{proof}
This is the main result of \cite{T}.
\end{proof}

Being defined as in (\ref{HH}),
the higher Hochschild cohomological complex is not equipped with
an explicit $e_n$-algebra structure
(whereas the $e_{n-1}$-algebra structure can be made explicit,
see \cite{GTZ}).
In particular, the $L_\infty$-structure 
on $\LL(\HH{n-1}{M})$ is rather implicit.

But for $n=2$ the higher Hochschild cohomological complex
is the usual Hoch\-schild cohomological complex
and it is 
equipped with a Lie bracket due to Gerstenhaber \cite{Ger}.
If a $\fm{2}$-algebra  $A$ acts on  an  algebra $M$,
one may build an explicit $L_\infty$-morphism
from $\LL(A)$ to the  Hochschild cohomological complex
of $M$ equipped with the Gerstenhaber bracket.
Note that the latter is a dg-Lie algebra, it has
no higher $L_\infty$-operations.
It would be interesting to generalize this construction 
for higher dimensions.

Let $A$ be a $\fm{2}$-algebra  acting on a $\fm{1}$-algebra $M$.
Define a chain  in the complex dual to $C^*(\SCO^S, \Sh{\varepsilon^{\vee}})$ depending on $k$ elements of $A$ and $l$ elements of $M$.
Let $B_2=\{x\in\R{2}{>0}| |x|<1\}$ and $B_1=\{x\in\R{1}{=0}| |x|<1\}$.
Define  chain $\tilde{c}$ by
\begin{equation*}
\begin{split}
&\tilde {c}(a_1,\dots, a_k;m_1,\dots ,m_l)=\\
&\qquad \qquad [\Cs{B^2}{\set{k}}]\otimes_{\Sigma_k} (a_1\otimes\cdots a_k)
\cup  [\Cs{B^1}{\set{l}}]\otimes_{\Sigma_l} (m_1\otimes\cdots m_l),
\end{split}
\end{equation*}
where  $[\Cs{B^2}{\set{k}}]]$ and $ [\Cs{B^1}{\set{l}}]$ are cycles in
$C_*(\CS{\R{2}{\ge 0}}{S})$ presented by the configuration space of distinct points
lying in  $B_2$ and $B_1$.
Consider the fiberwise closure of this chain with respect
to the projection of configuration spaces, which forgets closed points.
Take its intersection with the subset consisting of configurations
with all closed points lying over the origin. 
Denote the resulting chain by $ c(a_1,\dots, a_k;m_1,\dots ,m_l)$.
The convolution of this chain with action
defines a map from $A^{\otimes k}\otimes M^{\otimes l}$
to $M$, which is symmetric in $A$'s.
Thus it defines a map
\begin{equation}
c\colon A^{\otimes k}[l+2k-2] \to \bigoplus\nolimits_l\Hom(M^{\otimes l }, M).
\label{q-lie}
\end{equation}

\begin{prop}
For  a $\fm{2}$-algebra  $A$ acting on a $\fm{1}$-algebra $M$ 
map (\ref{q-lie}) defines a $L_\infty$-morphism
from $\LL(A)$ to the Hochschild cohomological complex of $M$.
\label{kont}
\end{prop}

\begin{proof}
This is the Theorem from \cite[6.4]{K} slightly rephrased.
\end{proof}

Note that the logic of this construction is similar to Proposition 8 from \cite{W}.

\section{Propagator approach}

\subsection{Propagator}
\label{Prop}

Consider the stratified space $\SCO^{2}$ defined in Subsection \ref{SC}.
One may see, that this is the higher dimensional
generalization of the "eye" from \cite{K}.
The stratum with two closed points is the 
 interior of the "eye",
two strata  with one open and one closed point
 are "eyelids", which are hemispheres, the stratum with two open points 
is the "eye corner(s)", which is a $(n-2)$-dimensional sphere.

The manifold with corners  $\SCO^{2}$
has two connected components of the boundary. The first one consists of two
"eyelids" glued by the "eye corner(s)".
The second one is the "iris", the $(n-1)$-dimensional sphere, which magnifies
collisions of two closed points.

The following definition is a straightforward
high-dimensional generalization  of the differential of the angle map from \cite[6.2]{K}. 

\begin{definition}
A $n$-propagator is a smooth closed differential $(n-1)$-form on $\SCO^{2}$,
such that 
\begin{enumerate}
\item its restriction on the "lower eyelid" is zero and
\item its restriction on the "iris" is the standard volume form on the sphere.
\end{enumerate}
\label{propagator}
\end{definition}

Consider some examples of propagators. We define a differential form on the interior
of the "eye" and leave to the reader to check that it continues
to the boundary. The interior consists of pairs of distinct points in $\R{n}{>0}$
modulo the  $\Dil{n-1}$ action, that is $\Scs{\R{n}{> 0}}{\set{2}}/\Dil{n-1}$.
Denote such a pair by $(s,t)$. When $t$ tends to the boundary $\R{n-1}{=0}\subset\R{n}{\ge 0}$,
the pair $(s,t)$ tends to the "lower eyelid". 

\setcounter{example}{-1}

\begin{example}
The first example is a high-dimensional generalization of the propagator used in \cite{K}.

For a pair $(s,t)$ as above denote by $\overline{t}\in \R{n}{<0}$ the image
of the reflection of $t$ with respect to the boundary hyperplane.  
Consider two maps from $\Scs{\R{n}{> 0}}{\set{2}}$ to $S^{n-1}$
which send $(s,t)$  to directions given by vectors $s-t$ and $s-\overline{t}$.
The difference between pullbacks of the standard volume form on the sphere
under the first and the second maps is a $\Dil{n-1}$-invariant
closed $(n-1)$-form  on $\Scs{\R{n}{> 0}}{\set{2}}$. Its continuous extension to the boundary satisfies 
conditions of the Definition \ref{propagator}. Denote this propagator by
$\Pro_n^0$.
\end{example}

\renewcommand*{\theexample}{\alph{example}}
\setcounter{example}{10}
\begin{example}
\label{ex}
The previous example is  the first in a series.
 
For $k\in \mathbb{N}$ consider the embedding $\R{n}{}\hookrightarrow\R{n+k}{}$ as the coordinate plane. 
For any point, $t\in \R{n}{>0}\subset \R{n}{}$ denote by $S_t$ the only $k$-sphere in $\R{n+k}{}$,
which contains $t$, has its center on the plane $\R{n-1}{=0}\subset\R{n}{}$ and 
lies in the plane perpendicular to this plane. 
Consider the space of triples $(s, t, p)$, where $s, t \in \R{n}{>0}\subset \R{n}{}$, $s\neq t$ and
$p\in S_t$.
Denote by $\varpi$ the projection from this space to the configuration space of two distinct ordered points
$\Cs{\Hp}{\set{2}}$, which forgets the third term of the triple. 

On the configuration space $\Cs{\R{n+k}{}}{\set{2}}$ 
of two distinct points of $\R{n+k}{}$
consider the standard differential $(n+k-1)$-form $\mathfrak{v}$,
which is the pullback of the standard volume form of the $(n+k-1)$-sphere
under the projection given by the direction of the vector connecting two points.
Denote by the same letter the $(n+k-1)$-form on the space of triples as above,
which is the pullback of $\mathfrak{v}$ under the map which forgets the middle term of the triple.
Denote by $\Pro_n^k$ the differential $(n-1)$-form on  $\Cs{\Hp}{\set{2}}$
given by the integration of $\mathfrak{v}$ along the projection $\varpi$:
\begin{equation}
\Pro_n^k(s,t)=\int_\varpi \mathfrak{v}(s,p).
\label{pr-formula}
\end{equation}

One may see, that this form is closed,  invariant under $\Dil{n-1}$ 
and allows the smooth extension to 
$\SCO^{2}$, which obeys conditions of Definition \ref{propagator},
that is it is a propagator.

In this example one may replace the sphere with 
any figure in the Euclidean space, which does not contain the origin 
and once intersects any ray from the origin. For example, one may take
an ellipsoid.  If one consider a family of ellipsoids in $\R{n+k}{}$, or any other figures,
which tends to the cylinder over $S^l$, the corresponding propagator tends
to $\Pro_n^{l}$.
\end{example}

\subsection{Action from a propagator}

In subsection \ref{weyl} following \cite{W} we construct the Weyl algebra 
over the operad  $\fm{n}$.
Given a propagator one may, in the same manner, construct an action of
this  $\fm{n}$-algebra on the polynomial algebra considered as a  $\fm{n-1}$-algebra. 

Let $U$ be a $\mathbb{Z}$-graded finite-dimensional vector space over the base field $\Bbbk$
of characteristic zero containing $\mathbb R$ and $U^{\vee}$ be the dual space.
Then $V=U\oplus U^{\vee}[1-n]$ is naturally equipped with the perfect skew-symmetric
pairing of degree $1-n$. By Definition \ref{wna} this data gives us
the Weyl $\fm{n}$-algebra $\We{n}{U\oplus U^{\vee}[1-n]}$ with the underlying complex 
$\Bbbk[U\oplus U^{\vee}[1-n]]$.

Consider the triple  $(\Bbbk[U\oplus U^{\vee}[1-n]], \Bbbk[U], \varepsilon)$ 
where $\varepsilon\colon \Bbbk[U\oplus U^{\vee}[1-n]]\to \Bbbk[U]$ is the natural map,
which sends all generators from $ U^{\vee}[1-n]$ to zero.

As in  Definition \ref{sc} denote by  $\Sh{\varepsilon^{\vee}}$ the combinatorial
sheaf over $\SCO^S$ associated with $\varepsilon^{\vee}$.
There is a natural map 
\begin{equation}
 \Sh{\varepsilon^{\vee}}\to (\Bbbk[U\oplus U^{\vee}[1-n]]^{\vee})^S
\label{map0}
\end{equation}
from $\Sh{\varepsilon^{\vee}}$ to the constant sheaf   given by the augmentation.

Fix a $n$-propagator $\mathfrak{p}$. 

For any two-element subset $\{i,j\}\subset S$ 
denote by $p_{ij}\colon \SCO(S)\to \SCO(\set{2})$ the map that forgets
all points except ones marked by $i$ and by $j$. 
Denote by $\mathfrak{p}_{ij}$ the pullback of $\mathfrak{p}$ under projection $p_{ij}$.
Let $\alpha$ be an element of endomorphisms of
$\Bbbk[U\oplus U^{\vee}[1-n] ]^{\otimes S} \tens{Aut (S)}C^*(\SCO(S))$ (where $C^*(-)$ is the de Rham complex)
given by 
$$
\alpha= \sum_{i,j\in S}\partial_\omega^{ij}\wedge \mathfrak{p}_{ij} ,
$$
where $\partial_\omega^{ij}$ is the operator $\partial_\omega$ 
applied to the $i$-th and $j$-th factors, where the operator $\partial_\omega$
is given by (\ref{omega}) for the standard bilinear form of degree $1-n$
on $ U\oplus U^{\vee}[1-n]$.

As in Proposition \ref{prod} consider the map  
\begin{equation}
\Bbbk[U\oplus U^{\vee}[1-n]]^{\otimes S}\to  C^*(\SCO^S)\otimes \Bbbk[U\oplus U^{\vee}[1-n]]
\label{prop}
\end{equation}
given by the composition of $\exp(\alpha)$ and $\mu$.

Composing the map dual to (\ref{prop}) 
\begin{equation}
\begin{aligned}
(\Bbbk[U\oplus U^{\vee}[1-n]])^{\vee}\to C^*(\SCO^S)\otimes (\Bbbk[U\oplus U^{\vee}[1-n]]^{\vee})^{\otimes S} =\\ C^*(\SCO^S,(\Bbbk[U\oplus U^{\vee}[1-n]]^{\vee})^{\otimes S} )
\end{aligned}
\label{map}
\end{equation}
 with 
$$
\varepsilon^{\vee} \colon \Bbbk[U]^{\vee}\to (\Bbbk[U\oplus U^{\vee}[1-n]])^{\vee}
$$
we get the map
\begin{equation}
 \Bbbk[U]^{\vee}\to C^*(\SCO^S,(\Bbbk[U\oplus U^{\vee}[1-n]]^{\vee})^{\otimes S} ).
\label{map2}
\end{equation}

\begin{prop}
For a $n$-propagator $\mathfrak{p}$ the map (\ref{map2}) uniquely factors through the map
induced by the map of sheaves (\ref{map0})
and the resulting map
$$
\Bbbk[U]^{\vee}\to C^*(\SCO^S, \Sh{\varepsilon^{\vee}})
$$
 defines an action of $\We{n}{U\oplus U^{\vee}[1-n]}$ on on the polynomial algebra $\Bbbk[U]$ considered as a  $\fm{n-1}$-algebra.
\label{prop-action}
\end{prop}

\begin{proof}
The existence of the unique factorization follows from the first property of the propagator from Definition \ref{propagator}. The second property guarantees the second condition of Definition \ref{sc}.
\end{proof}

Note that to get such an action one could start with a gadget more general, than the propagator:
one can use a closed form on  $\SCO^2$ taking values in $V^{\vee}\otimes V^{\vee}$, which vanishes
on the "lower eyelid" and when restricted on the "iris" equals to the standard volume form
multiplied by the standard bilinear form of degree $1-n$ on $V$.

\subsection{Formality}
Proposition \ref{prop-action} gives an action
of $\We{n}{U\oplus U^{\vee}[1-n]}$ on  the polynomial algebra $\Bbbk[U]$ considered as an  $\fm{n-1}$-algebra.
By Proposition \ref{thomas}, this action gives a morphism of $e_n$-algebras
\begin{equation}
\We{n}{U\oplus U^{\vee}[1-n]}\to \HH{n-1}{\Bbbk[U]}.
\label{act-p}
\end{equation}

\begin{theorem}
Map (\ref{act-p}) is a quasi-isomorphism.
\label{iso}
\end{theorem}

\begin{proof}
Let us first evaluate $\HH{n-1}{\Bbbk[U]}$. By Proposition \ref{commutative}
$\int_{S^{n-2}} \Bbbk[U]=\Bbbk[U\oplus U[n-2]]$. From (\ref{HH}) we get  
$$
 \HH{n-1}{\Bbbk[U]}=\RHom{\int_{S^{n-2}}\Bbbk[U]}{\Bbbk[U]}{\Bbbk[U]}=\Bbbk[U\oplus U^{\vee}[1-n]].
$$
The commutative algebra structure on the $\Bbbk[U\oplus U^{\vee}[1-n]]$
comes from the $e_{n-1}$-algebra structure on $\HH{n-1}{\Bbbk[U]}$
by the very definition of the latter. But 
by Proposition \ref{obl}
$\We{n}{U\oplus U^{\vee}[1-n]}$
as an $e_{n-1}$-algebra is also a free commutative algebra generated by $U\oplus U^{\vee}[1-n]$.
Since (\ref{act-p}) is a morphism of $e_{n-1}$-algebras, to prove
the statement we need to show that (\ref{act-p}) defines an isomorphism 
on generators. This may be easily checked by the explicit
definition of the morphism given before Proposition \ref{thomas}.   
\end{proof}

Combining this theorem with Proposition \ref{poiss} we
get the following corollary.

\begin{cor}
Map (\ref{q-lie}) defines a quasi-isomorphism
between $\LL( \HH{n-1}{\Bbbk[U]})$
and  the Poisson Lie algebra
of $(\Bbbk[U\oplus U^{\vee}[1-n]],\, \omega)$
 as $L_\infty$-algebras, where $\omega$ is
the standard bilinear form on $U\oplus U^{\vee}[1-n]$.
\label{formality}
\end{cor}

Being combined  with Proposition \ref{kont},
it gives us the main result of \cite{K}
about the quasi-isomorphism between polyvector fields and the Hochschild
cohomological complex of polynomial algebra.
In that paper only propagator $\Pro_2^0$, which is defined in Subsection \ref{Prop}, 
is used,
although it is mentioned there that any other propagator also does the job.
It is known, that coefficients of the above formality quasi-isomorphism
for this propagator
are given by  integrals similar to multiple zeta values (see e.g. \cite{RW}), in particular,
they are conjecturally irrational.

Results of the next section lead us to the following conjecture.

\begin{conj}
The formality morphism given by Corollary \ref{formality} 
with propagator $\Pro_n^k$ for $k>0$ has rational coefficients.
\label{conj}
\end{conj}

In the next section, we shall prove this conjecture in the two-dimensional case
(Corollary \ref{corollary}).

\section{Factorization complex approach}

\subsection{Factorization complex of a sphere}

By Proposition \ref{fubini} for an invariant $\fm{n}$-algebra $A$ and 
an oriented closed manifold $M^k$ of dimension $k<n$ the complex
$\int_{M^k} A$ is a $\fm{n-k}$-algebra. 
If $A$ is a Weyl $n$-algebra the natural candidate for $\int_{M^k} A$
is a Weyl $(n-k)$-algebra again. 

\begin{conj}
For an oriented closed $k$-manifold $M^k$ and a Weyl $n$-algebra $\We{n}{V}$, where $n>k$, the
factorization complex $\int_{M^k} \We{n}{V}$
is quasi-isomorphic to $\We{n-k}{V\otimes H_*(M^k)}$
as a $\fm{n-k}$-algebra, where $H_*(M^k)$
is the homology of $M^k$ is negatively graded and equipped with the Poincar\'e pairing.
\end{conj}

The factorization complex on a sphere is of particular interest,
see e.~g. \cite[Proposition 6.2]{GTZ}.

\begin{prop}
For a natural $k<n$ the $\fm{n-k}$-algebra $\int_{S^k} \We{n}{V} $
is quasi-iso\-mor\-phic to $\We{n-k}{V\oplus V[k]}$,
where $V\oplus V[k]$ is  equipped with the natural perfect pairing
of degree $1-n+k$. 
\label{sphere}
\end{prop}

\begin{proof}
Fix a point $p\in S^k$.
Denote by
$[\Cs{S^k\setminus p}{S}]$  the cycle in 
$C_*(\CS{S^k}{S})$ presented by the configuration space of distinct points
of $S^k$  labeled by $S$
 distinct from $p$. Denote by $[p]$ the cycle in 
$C_*(\CS{S^k}{S})$ presented by point $p$.
Let $[x\otimes y]$ be a monomial representing
an element of   $\We{n-k}{V\oplus V[k]}$, where $x\in \Bbbk[V]$ and $y\in \Bbbk[V[k]]$ are  monomials.
Let $y=l_i\cdots l_i$ is a factorization of $y$ into linear factors.

Define a map from  $\We{n-k}{V\oplus V[k]}$
to the factorization complex
$\int_{S^k} \We{n}{V} $
by
$$
[x\otimes y] \mapsto [p]\otimes x\cup [\Cs{S^k\setminus p}{\set{i}}]\otimes_{\Sigma_i} l_1\otimes\dots\otimes l_i
$$
The image  is closed because a Weyl $n$-algebra is commutative
as Weyl $k$-algebra by Proposition \ref{obl} and the cycle above 
is the standard cycle representing a class in the factorization complex
of a polynomial algebra
generalizing the standard class for the
Hochschild complex \cite[Proposition 1.3.12]{Lo} (compare with \cite[Definition 2]{GTZ2}). 
It is easy  to show by the direct calculation,
that this map respects the $\fm{n-k}$-algebra structure.
The crucial point here is using relations (\ref{reduced}).
\end{proof}

\subsection{Action}
\label{action-sec}

As in Subsection \ref{Prop} for
$k\in \mathbb{N}$
define the embedding $\R{n}{}\hookrightarrow\R{n+k}{}$ as the coordinate plane
and consider the projection
\begin{equation}
\R{n+k}{}\to \R{n}{\ge 0},
\label{projection}
\end{equation}
which sends a point to the only intersection
with $\R{n}{}\hookrightarrow\R{n+k}{}$ of the only $k$-sphere in $\R{n+k}{}$,
which contains this point, has its center on the plane $\R{n-1}{=0}\subset\R{n}{}$ and 
lies in the plane perpendicular to this plane.

For an invariant $\fm{n+k}$-algebra $A$ the product of the operad
gives a $\Dil{n-1}$-invariant map
\begin{equation}
A^{\vee}\to \bigoplus\limits_{S} C^*(\Cs{\R{n+k}{}}{S})\tens{Aut(S)} {A^{\vee}}^{\otimes S}
\label{product}
\end{equation}
Denote by ${A^{\vee}}^\otimes$ the locally constant sheaf over
$\coprod_S \Cs{\R{n+k}{}}{S}$ with the fiber equal to $ {A^{\vee}}^{\otimes S}$.
Then the right side of (\ref{product}) is $H^*(\coprod_S \Cs{\R{n+k}{}}{S}, {A^{\vee}}^\otimes)$

Now we need to take the push-forward of the factorization
sheaf with respect to the map (\ref{projection}). 
The following construction, being a generalization of Proposition \ref{fubini},
could be formulated in terms of the relative factorization complex.
But the relative factorization complex is a cosheaf rather
than a sheaf. That is why it is more convenient to work with the linear dual
thing.
Define the dual relative factorization complex $(\int_{\R{n+k}{}/\R{n}{\ge 0}}A)^{\vee}$
of $A$ of the map (\ref{projection}),
which is a complex of sheaves over $\R{n}{\ge 0}$,
in analogy with (\ref{coend}) as the limit of the diagram
\begin{equation}
\begin{tikzcd}
\bigoplus\limits_{S'} C^*_{\R{n}{\ge 0}}(\CS{\R{n+k}{}}{S'})\tens{Aut(S')} {A^{\vee}}^{\otimes S'}\arrow[d]\\
\bigoplus\limits_{i\colon S'\to S}C^*_{\R{n}{\ge 0}}(\Cs{M}{S})\tens{Aut(S)}\bigotimes\limits_{s\in 
S} 
(\fm{k}(i^{-1}s)\tens{Aut(i^{-1}s)} A^{\otimes 
{(i^{-1}s)}})^{\vee} \\
\bigoplus\limits_{S}C^*_{\R{n}{\ge 0}}(\Cs{M}{S})\tens{Aut(S)} {A^{\vee}}^{\otimes S} \arrow[u]
\end{tikzcd}
\label{end}
\end{equation}
where the operad $\fm{k}$ coacts along fibers of (\ref{projection}),
this coaction is trivial over $\R{n-1}{=0}\subset \R{n}{\ge 0}$.

In the diagram
$$
\begin{tikzcd}
C^*(\CS{\R{n+k}{}}{S})\tens{Aut(S)} {A^{\vee}}^{\otimes S}  & C^*(\R{n}{\ge 0}, (\int_{\R{n+k}{}/\R{n}{\ge 0}}A)^{\vee}) \arrow[l]\\
\qquad \qquad \qquad A^{\vee}\arrow[u] = (\int_{\R{n+k}{}} A)^{\vee} \arrow[ru, dashed]&
\end{tikzcd}
$$
the isomorphism in the bottom row is given by Proposition \ref{odisk},
the horizontal arrow is the embedding in the top term 
of the diagram (\ref{end}) and the vertical is dual to the action of $\fm{n+k}$
operad on $A$. The  dashed arrow, which makes the diagram commutative,
exists because relations  dual to (\ref{end}) are contained in 
the relation defining the big factorization complex $\int_{\R{n+k}{}} A$.
The dashed arrow is $\Dil{n-1}$-invariant and it gives us a map
\begin{equation}
A^{\vee}\to C^*(\SCO^S, (\int_{\R{n+k}{}/\R{n}{\ge 0}}A)^{\vee})
\label{action2}
\end{equation}

Recall that for a $\fm{n}$-algebra $A$ and $m<n$ we denote by $\obl{n}{m} A$
the $\fm{m}$-algebra with the same underlying complex as $A$
and the operadic structure induced by the natural map of operads $\fm{m}\to \fm{n}$.
A particular case of the following proposition was crucial for the second part of \cite{WK}.

\begin{prop}
For an invariant $\fm{n+k}$-algebra $A$ consider the triple
$$(\int_{S^k} A, \,\obl{n+k}{n-1} A,\, \varepsilon),$$ where
$\int_{S^k} A$ is an $\fm{n}$-algebra due to Proposition \ref{fubini}
and $\varepsilon\colon \int_{S^k} A \to A$ is the natural map
induced by the map from the sphere to a point.
Then the complex $\mathcal{F}_\varepsilon^{\vee}$ is quasi-isomorphic
to the complex  $(\int_{\R{n+k}{}/\R{n}{\ge 0}}A)^{\vee}$
defined above and  map (\ref{action2})
defines an action of $\int_{S^k} A$ on $\obl{n+k}{n-1} A$
in the sense of Definition \ref{def-action}.
\label{main}
\end{prop}
\begin{proof}
One may see that the complex $(\int_{\R{n+k}{}/\R{n}{\ge 0}}A)^{\vee}$ is constructible
with respect to the stratification (\ref{strat})
with fibers and the restriction map equal to ones of  $\mathcal{F}_\varepsilon^{\vee}$,
what follows that they are quasi-isomorphic.
Axioms of the action (Definition \ref{def-action})
follow from the very definition of $\fm{n+k}$-algebra and the one of the factorization complex.
\end{proof}

Applying Proposition \ref{thomas} to the action constructed
above we get for any $\fm{n+k}$-algebra $A$ a natural map of $\fm{n}$-algebras.
\begin{equation}
\int_{S^k} A \to \HH{n-1}{\obl{n+k}{n-1} A},
\label{HH-act}
\end{equation}
where $\HH{n-1}{\cdot}$ is the higher Hochschild cohomology, see Subsection \ref{HHC}.

\subsection{Formality via the factorization complex}

We are going to show that for Weyl $(n+k)$-algebras map (\ref{HH-act})
is a quasi-isomorphism.

Recall, that by Proposition \ref{sphere} the $\fm{n}$-algebra
$\int_{S^k} \We{n+k}{V}$ is isomorphic to $\We{n}{V\oplus V[k]}$,
and $\obl{n+k}{n-1} \We{n+k}{V}$ is isomorphic to  $\Bbbk[V]$
by Proposition \ref{obl}.

\begin{prop}
For a Weyl $(n+k)$-algebra $\We{n+k}{V}$ the action of
$\int_{S^k} \We{n+k}{V}=\We{n}{V\oplus V[k]}$
on $\obl{n+k}{n-1} \We{n+k}{V}=\Bbbk[V]$
defined by Proposition \ref{main}
is isomorphic to the action defined by Proposition
\ref{prop-action} for propagator $\Pro_n^k$ given by  (\ref{pr-formula}).
\label{last}
\end{prop}
\begin{proof}
Substituting quasi-isomorphism from Proposition \ref{sphere}
to the definition of the action from Proposition \ref{main}
we immediately get the statement.  
\end{proof}

Proposition \ref{thomas} gives
a map of $\fm{n}$-algebras
\begin{equation}
\int_{S^k} \We{n+k}{V} \to \HH{n-1}{\obl{n+k}{n-1} \We{n+k}{V}}=\HH{n-1}{\Bbbk[V]}
\label{WHH}
\end{equation}

\begin{theorem}
The map (\ref{WHH}) is  a quasi-isomorphism with rational coefficients.
\end{theorem}

\begin{proof}
Combining Theorem  \ref{iso} with Proposition \ref{last}
we get the statement. To see that coefficients are rational
note that the action is given by the product in 
the operad $\fm{n+k}$. Thus
 coefficients of the action (\ref{action2})
are given by integration of an integral cocycle by an integer cycle. 
\end{proof}

Combining this theorem with Proposition \ref{poiss} and Proposition \ref{sphere} we get the following corollary.
\begin{cor}
Map (\ref{WHH}) gives a quasi-isomorphism between two  $L_\infty$-algebras: 
the higher Hochschild cohomology Lie algebra
$\LL(\HH{n-1}{\Bbbk[V]})$  and the Poisson Lie algebra
of $(\Bbbk[V\oplus V^{\vee}[n]], \omega)$,
where $\omega$ is the standard bilinear form on $[V\oplus V^{\vee}[n]$
with rational coefficients.
\end{cor}

Let $V$ be a vector space with a perfect symmetric pairing.
Applying this corollary to Weyl 3-algebra $\We{3}{V[1]}$ and
combining it with Corollary \ref{formality} 
and Proposition \ref{last}
we get Conjecture \ref{conj} in dimension two.

\begin{cor}
For a vector space  $V$ 
the formality morphism given by Corollary \ref{formality}
with propagator $\Pro_2^k$ for $k>0$
gives a quasi-isomorphism between
$\LL(\HH{1}{\Bbbk[V]})$ and the Lie algebra of polyvector fields $\Bbbk[V\oplus V^{\vee}[1]]$ with rational coefficients.
\label{corollary}
\end{cor}

\bibliographystyle{alpha}
\bibliography{swiss}

\end{document}